\documentclass[12pt]{amsart}
\usepackage{amssymb,amsmath,amsthm}
\newcommand{\leg}[2]{\genfrac{(}{)}{}{}{#1}{#2}}

\newtheorem{theorem}{Theorem}

\theoremstyle{remark}

\newtheorem{remark}{Remark}

\numberwithin{theorem}{section} \numberwithin{equation}{section}

\begin{document}

\title[Congruences for Broken $k$-Diamond Partitions]
{Congruences for Broken $k$-Diamond Partitions}
    \author{Marie Jameson}

    \address{Department of Mathematics and Computer Science, Emory University,
    Atlanta, Georgia 30322}
\email{mjames7@emory.edu}

\begin{abstract}
We prove two conjectures of Paule and Radu from their recent paper on broken $k$-diamond partitions.
\end{abstract}
\maketitle
\noindent

\section{Introduction and Statement of Results}
In \cite{AndrewsPaule}, Paule and Andrews constructed a class of directed graphs called broken $k$-diamonds, and they used them to define $\Delta_k(n)$ to be the number of broken $k$-diamond partitions of $n$. They noted that the generating function for $\Delta_k(n)$ is essentially a modular form. More precisely, if $k \geq 1,$ then
\begin{equation}\label{genfun}
\sum_{n=0}^\infty \Delta_k(n)q^n = q^{(k+1)/12}\frac{\eta(2z)\eta((2k+1)z)}{\eta(z)^3\eta((4k+2)z)},
\end{equation}
where $\eta(z)$ is Dedekind's eta function \[\eta(z) = q^{1/24}\prod_{n=1}^\infty\left(1-q^n\right), \qquad \left(q = e^{2\pi i z} \right).\]

One can show various congruences for $\Delta_k(n)$ for $n$ in certain arithmetic progressions.  For example, Xiong \cite{Xiong} proved congruences for $\Delta_3(n)$ and $\Delta_5(n)$ which had been conjectured by Paule and Radu in \cite{PauleRadu}. In particular, he showed that
\begin{equation}\label{conj1}
\prod_{n=1}^\infty(1-q^n)^4(1-q^{2n})^6 \equiv 6\sum_{n=0}^\infty \Delta_3(7n+5)q^n \pmod{7}
\end{equation}

In this note, we prove the remaining two conjectures in \cite{PauleRadu}. First, we use \eqref{conj1} to prove the following statement (which is denoted Conjecture 3.2 in \cite{PauleRadu}).
\begin{theorem}\label{conj2}
For all $n \in \mathbb{N},$ we have that \[\Delta_3(7^3n+82) \equiv \Delta_3(7^3n+229)\equiv \Delta_3(7^3n+278) \equiv \Delta_3(7^3n+327) \equiv 0 \pmod{7}.\]
\end{theorem}
Now, recall that the weight $k$ Eisenstein series (where $k\geq 4$ is even) are given by \[E_k(z):=1 - \frac{2k}{B_k}\sum_{n=1}^\infty\sigma_{k-1}(n)q^n,\] where $B_k$ is the $k$th Bernoulli number, and $\sigma_{k-1}(n):=\sum_{d|n}d^{k-1}.$ Also define \begin{equation}\sum_{n=0}^\infty c(n)q^n :=E_4(2z)\prod_{n=1}^\infty(1-q^n)^8(1-q^{2n})^2.\end{equation}

The coefficients $c(n)$ are of interest here because they are related to broken $k$-diamond partitions in the following way (as conjectured in \cite{PauleRadu} and proved in \cite{Xiong}):
\begin{equation}\label{conj3}
c(n) \equiv 8\Delta_5(11n+6) \pmod{11}.
\end{equation}
Here we prove the last remaining conjecture of Paule and Radu (which is Conjecture 3.4 of \cite{PauleRadu}). More precisely, we have the following theorem.

\begin{theorem}\label{cor}
For every prime $p\equiv 1 \pmod{4},$ there exists an integer $y(p)$ such that 
\[c\left(pn + \frac{p-1}{2}\right) + p^8c\left(\frac{n-(p-1)/2}{p}\right) = y(p)c(n)\] for all $n \in \mathbb{N}.$
\end{theorem}

\begin{remark}
Theorem \ref{cor} follows from a more technical result (see Theorem \ref{thm} which is proved in Section \ref{prf}).
\end{remark}

\begin{remark}
As noted in \cite{PauleRadu}, one can combine \eqref{conj3} with Theorem \ref{cor} to see that for every prime $p\equiv 1 \pmod{4}$ and $n \in \mathbb{N}$ we have
\[\Delta_5\left((11n+6)p - \frac{p-1}{2}\right) + p^8\Delta_5\left(\frac{11n+6}{p} + \frac{p-1}{2p}\right) \equiv y(p)\Delta_5(11n+6)\pmod{11}.\]
\end{remark}

To prove Theorems \ref{conj2} and \ref{cor}, we make use of the theory of modular forms. In particular, we shall make use of the $U$-operator, Hecke operators, the theory of twists, and a theorem of Sturm. These results are described in \cite{Ono}. We shall freely assume standard definitions and notation which may be found there.
\section{Proof of Theorem \ref{conj2}}

First we consider the form $\eta(3z)^4\eta(6z)^6.$ By Theorems 1.64 and 1.65 in \cite{Ono}, we have that $\eta(3z)^4\eta(6z)^6\in S_5\left(\Gamma_0(72), \leg{-1}{\bullet}\right).$ Note from \eqref{conj1} that \[\eta(3z)^4\eta(6z)^6 \equiv 6\sum_{n=0}^\infty \Delta_3(7n+5)q^{3n+2} \pmod{7}.\] It follows that  \[f(z):=\eta(3z)^4\eta(6z)^6 \mid U_7 \equiv 6\sum_{n=0}^\infty \Delta_3(7^2n+33)q^{3n+2} \pmod{7}.\] Here, $U_d$ denotes Atkin's $U$-operator, which is defined by \[\sum_{n=0}^\infty a(n)q^n \mid U_d = \sum_{n=0}^\infty a(dn)q^n\] for $d$ a positive integer. By the theory of the $U$-operator (see Proposition 2.22 in \cite{Ono}), it follows that $f(z) \in S_5\left(\Gamma_0(504), \leg{-1}{\bullet}\right).$ Now if we define $b(n)$ by \[f(z) := \sum_{n=0}^\infty b(n)q^n,\] then our goal is to show that \[b(21n + 5) \equiv b(21n+14)\equiv b(21n + 17) \equiv b(21n+20) \equiv 0 \pmod{7}.\]

In order to prove the desired congruence, consider the Dirichlet character $\psi$ defined by $\psi(d) :=\leg{d}{7},$ we may consider the $\psi$-twist of $f$, which is given by \[f_\psi(z) := \sum_{n=0}^\infty \psi(n)b(n)q^n.\] By Proposition 2.8 of \cite{Ono}, we have that $f_\psi(z) \in S_5\left(\Gamma_0(24696), \leg{-1}{\bullet}\right).$

Then consider \[f(z) - f_\psi(z) = \sum_{n=0}^\infty \left(1-\leg{n}{7}\right)b(n)q^n \in S_5\left(\Gamma_0(24696), \leg{-1}{\bullet}\right).\] In fact, $f(z) - f_\psi(z) \equiv 0 \pmod{7}.$ This follows from a theorem of Sturm (see Theorem 2.58 in \cite{Ono}), which states that $f(z) - f_\psi(z) \equiv 0 \pmod{7}$ if its first $23520$ coefficients are 0 (mod 7) (which can be checked using a computer). Thus we have that \[\left(1-\leg{n}{7}\right)b(n) \equiv 0 \pmod{7}\] for all $n$, and thus \[b(21n + 5)\equiv b(21n+14) \equiv b(21n + 17) \equiv b(21n+20) \equiv 0 \pmod{7}\] for all $n \in \mathbb{N},$ as desired.

\section{Proof of Theorem \ref{cor}} \label{prf}

\subsection{Preliminaries} \label{prelim}
Let us first recall the Hecke operators and their properties.  If $f(z) = \sum_{n=0}^\infty a(n)q^n \in M_k(\Gamma_0(N),\chi)$ and $p$ is prime, the Hecke operator $T_{p,k,\chi}$ (or simply $T_p$ if the weight and character are known from context) is defined by \[f(z) \mid T_{p} := \sum_{n=0}^\infty \left( a(pn) + \chi(p)p^{k-1}a(n/p)\right)q^n,\] where we set $a(n/p)=0$ if $p \nmid n.$ It is important to note that $f(z) \mid T_{p} \in M_k(\Gamma_0(N),\chi).$

In order to prove the final statement of Theorem \ref{cor}, define \[g(z) =\sum_{n=0}^\infty b(n)q^n:= E_4(4z)\eta(2z)^8\eta(4z)^2 \in S_9\left(\Gamma_0(16),\leg{-4}{\bullet}\right)\] and note that $c(n) = b(2n+1).$ Thus we wish to show that for every prime $p\equiv 1 \pmod{4}$ there exists an integer $y(p)$ such that 
\[b\left(p(2n+1)\right) + p^8b\left(\frac{2n+1}{p}\right) = y(p)b(2n+1)\]
for all $n \in \mathbb{N}.$ By summing (and noting that $b(n)=0$ when $n$ is even) we see that this is equivalent to the statement that \[g(z) \mid T_{p} = y(p)g(z).\] That is, we need only show that $g(z)$ is an eigenform of the Hecke operator $T_{p}$ for all $p\equiv 1 \pmod{4}.$

To see this, we let $F$ be the weight 2 Eisenstein series given by \[F(z) := \frac{\eta(4z)^8}{\eta(2z)^4}= \sum_{n=0}^\infty \sigma_1(2n+1)q^{2n+1} \in M_2(\Gamma_0(4)),\] let $\theta_0(z)$ be the theta-function given by \[\theta_0(z) := \sum_{n=-\infty}^\infty q^{n^2},\] and let $h(z)$ be the normalized cusp form \[h(z) := \eta(4z)^6 = \sum_{n=1}^\infty a(n)q^n = q-6q^5 + 9q^9 + \cdots \in S_3\left(\Gamma_0(16),\leg{-4}{\bullet}\right).\] Then $h(z)$ is a modular form with complex multiplication (see Section 1.2.2 of \cite{Ono}), and for prime $p$ we have
\[a(p) = \begin{cases} 2x^2-2y^2 & p=x^2+y^2 \text{ with }x,y \in \mathbb{Z} \text{ and } x \text{ odd}\\ 0 & p \equiv 2,3 \pmod{4}.\end{cases}\]
Then we may define $f_1, f_2, f \in S_9\left(\Gamma_0(16),\leg{-4}{\bullet}\right)$ by 
\begin{align*}
f_1(z)  &= \sum_{n=0}^\infty d_1(n)q^n := E_4(4z)F(z) \left[4\theta_0^6(4z) - \theta_0^6(2z) + 4\theta_0^4(2z)\theta_0^2(4z) - 6\theta_0^2(2z)\theta_0^4(4z) \right] \\
f_2(z) &= \sum_{n=0}^\infty d_2(n)q^n := E_4(4z)F(2z)h(z)\\
f(z) &= \sum_{n=0}^\infty d(n)q^n := f_1(z) + 8i\sqrt{3}f_2(z).
\end{align*}
We prove the following theorem involving these forms.

\begin{theorem} \label{thm}
The forms $f(z)$ and $\overline{f(z)}$ are eigenforms of the Hecke operator $T_p$ for all primes $p$.  Furthermore we have that \[\mathbb{T}_g = \langle f, \overline{f} \rangle,\] where $\mathbb{T}_g$ is the subspace of $S_9\left(\Gamma_0(16),\leg{-4}{\bullet}\right)$ spanned by $g$ together with $g\mid T_p$ for all primes $p$.
\end{theorem}

\begin{proof}
First note that $f$ and $\overline{f}$ are eigenforms of the Hecke operator $T_p$ for all primes $p$.  To see this, note that there is a basis of Hecke eigenforms of the space $S_9\left(\Gamma_0(16),\leg{-4}{\bullet}\right)$. Also, both $f$ and $\overline{f}$ are eigenforms of $T_5$ with eigenvalue 258, one can compute that this eigenspace \[\mathrm{ker}\left(T_5 - 258\right)\] is 2-dimensional.  Finally, both $f$ and $\overline{f}$ are eigenforms of the Hecke operator $T_7,$ and they have different eigenvalues.

Now, note that \[g = \left( \frac{1}{2} + \frac{i}{2\sqrt{3}} \right) f +  \left( \frac{1}{2} - \frac{i}{2\sqrt{3}} \right) \overline{f}\] and thus $\mathbb{T}_g$ is a two-dimensional subspace of $\langle f, \overline{f}\rangle.$ Thus $\mathbb{T}_g = \langle f, \overline{f}\rangle,$ as desired.
\end{proof}
\subsection{Proof of Theorem \ref{cor}}
Suppose $p$ is a prime with $p\equiv 1\pmod{4}.$ Then we need only check that $f$ and $\overline{f}$ are eigenforms of $T_p$ with the \emph{same eigenvalue}. Since these eigenvalues are the coefficients of $q^p$ in the expansions of $f$ and $\overline{f}$, we need only show that \[d(p) = \overline{d(p)},\] i.e., $d(p) \in \mathbb{R}$.

Now, note that $d_2(p)=0$ since the coefficients of $f_2$ are only supported on indices that are congruent to 3 mod 4 by the descriptions of $E_4(4z), F(2z),$ and $h(z)$ given above. Thus $d(p) = d_1(p) \in \mathbb{R},$ as desired.
\nocite{*}
\bibliographystyle{plain}
\bibliography{biblio}

\end{document}